\documentclass[12pt]{amsart}

\usepackage{amsfonts, amssymb, amsmath, amsthm}
\newcommand{\R}{\mathbb{R}}
\newcommand{\N}{\mathbb{N}}
\newcommand{\Q}{\mathbb{Q}}
\newcommand{\Z}{\mathbb{Z}}

\newcommand{\T}{\mathbb{T}}

\newcommand{\C}{\mathcal{C}}

\newcommand{\cP}{\mathcal{P}}

\newcommand{\G}{\mathcal{G}}

\newcommand{\ilim}{\varprojlim} 
\newcommand{\dlim}{\varinjlim} 
\newcommand{\OP}{\Omega_{\Phi}}

\newcommand{\larr}{\left( \begin{array}{c}}
\newcommand{\rarr}{\end{array} \right) }

\newcommand{\lsqarr}{\left[ \begin{array}{c}}
\newcommand{\rsqarr}{\end{array} \right]}

\newcommand{\inv}{\varprojlim}

\newtheorem{theorem}{Theorem}
\newtheorem{Theorem}[theorem]{Theorem}
\newtheorem{cor}[theorem]{Corollary}

\newtheorem{lemma}[theorem]{Lemma} \newtheorem{prop}[theorem]{Proposition}

\newcommand{\Xmax}{X_{max}}

\newcommand{\E}{\mathcal{E}}
\newcommand{\coker}{\mbox{\rm coker\,}}
\newcommand{\freq}{\mbox{\rm freq}}
\newcommand{\Inf}{\mbox{\rm Inf}}
\newcommand{\im}{\mbox{\rm im}}
\newcommand{\Fi}{{\frak C}}
\newcommand{\np}{{n^\perp}}
\newcommand{\rk}{\mbox{\rm rk}}

\input xy

\xyoption{all}

\begin{document}

\title{Maximal equicontinuous factors and cohomology for tiling spaces}
\author{Marcy Barge} 
\address{Department of Mathematical Sciences,
Montana State University,
Bozeman, MT 59717-0240, USA}

\email{barge@math.montana.edu}
\author{Johannes Kellendonk}
\address{Universit\'e de Lyon, Universit\'e Claude Bernard Lyon 1,
Institut Camille Jordan, CNRS UMR 5208, 43 boulevard du 11 novembre
1918, F-69622 Villeurbanne cedex, France}
\email{kellendonk@math.univ-lyon1.fr}
\author{Scott Schmieding}
\address{Department of Mathematics,
University of Maryland,
College Park, MD 20742-4015,
USA}
\email{scott@math.umd.edu}
\date{\today}

\begin{abstract} We study the homomorphism induced on cohomology by the maximal equicontinuous factor map of a tiling space. We will see that this map is injective in degree one and
has torsion free cokernel. We show by example, however, that the cohomology of the maximal equicontinuous factor may not be a direct summand of the tiling cohomology. 
\end{abstract} 

\maketitle
\section{Introduction}
An effective procedure for studying the properties of a tiling, or point-pattern, $T$ of $\R^n$
is to consider the space $\Omega$ (called the {\em hull} of $T$) of all tilings that, up to translation, are locally indistinguishable from $T$. Dynamical properties of the action of $\R^n$ on $\Omega$, by translation, correspond to combinatorial properties of $T$. Regularity assumptions on $T$ guarantee that the dynamical system $(\Omega,\R^n)$
is compact and minimal. There is then a {\em maximal equicontinuous factor} $(\Omega_{max},\R^n)$, with semi-conjugacy $\pi:\Omega\to \Omega_{max}$; $\Omega_{max}$ is a compact abelian group on which $\R^n$ acts by translation and every equicontinuous factor of $(\Omega,\R^n)$ is a factor of $(\Omega_{max},\R^n)$. 

The relationship between the hull of a tiling and its maximal equicontinuous factor is of fundamental importance in certain aspects of tiling theory. 
For example, if $T$ is a (sufficiently well-behaved) distribution of ``atoms" in $\R^n$, the diffraction spectrum of $T$ is pure point (that is, $T$ is a perfect quasicrystal) if and only if the dynamical spectrum of $(\Omega,\R^n)$ is pure discrete (\cite{LMS}, \cite{D}), if and only if the 
factor map $\pi$ is $a.e.$ one-to-one (with respect to Haar measure, \cite{BargeKellendonk}).

In this article we study the properties of the homomorphism $\pi^*$ induced by the factor map $\pi$ in cohomology. This is directly motivated by a recent formulation of the {\em Pisot Substitution Conjecture} (\cite{BG}) in terms of the homological properties of $\pi^*$.
More generally, cohomology has long been a primary tool for understanding the structure of $\Omega$ (\cite{AP}, \cite{Sadun}, \cite{FHK}) and, at least for tilings with a non-trivial discrete component of dynamical spectrum, the pull-back of the cohomology of the maximal equicontinuous factor represents a sort of skeleton supporting the rest of the cohomology of $\Omega$.

The maximal equicontinuous factor of a tiling dynamical system is always a torus or solenoid so its cohomology (as a ring) is determined by its degree one cohomology. Consequently, our focus will be on $\pi^*$ in degree one (this is also the important degree for deformation theory (\cite{ClarkSadun06},\cite{Kellendonk08}) and the Pisot Substitution Conjecture), though we will have something to say in higher degree for projection patterns, in which cohomology is tied to complexity. A main result is that $\pi^*$ is injective in degree one with torsion-free cokernel. We will show by example, however,  that the first cohomology of the maximal equicontinuous factor is not necessarily a direct summand
of the first cohomology of $\Omega$. 

Let us say a few words about our methods. Given a continuous map $f:\Omega\to\T$ of the hull to the unit circle, and a vector $v\in\R^n$, there is a {\em Schwartzman winding number} $\tau(f)(v)$ of $f$ with respect to the $\R$-action
$T'\mapsto T'-tv$ on $\Omega$ in direction $v$ (\cite{Schwartzman}). This defines a functional, $v\mapsto
 \tau(f)(v)$, which depends only on the homotopy class of $f$. As the group of homotopy classes of maps of $\Omega$ to $\T$ is naturally isomorphic with the first integer cohomology $H^1(\Omega)$ of $\Omega$, $\tau$ provides a homomorphism 
from $H^1(\Omega)$ to ${\R^n}^*$. We will see that the degree one cohomology of the maximal equicontinuous factor can be identified with the group $\mathcal{E}$ of continuous eigenvalues of the $\R^n$-action on $\Omega$. Each eigenvalue, in turn, determines a functional on $\R^n$. With these identifications, $\tau\circ\pi^*$ is the identity, establishing that $\pi^*$ is injective in degree one.

The homomorphism $\tau$ described above is the degree one part of the Ruelle-Sullivan map (\cite{KellendonkPutnam06}). In the top degree, $n$, $\tau$ has an interpretation as
the homomorphism that assigns to each finite patch of a tiling $T$ its frequency of occurrence in $T$. The range of $\tau$ is then the {\em frequency module} $\freq(\Omega)$ of $\Omega$ and its kernel is the group $\Inf(\Omega)$ of {\em infinitessimals} with respect to a natural order on the top degree cohomology. In the special case of one-dimensional tilings, we have two related short exact sequences with $H^1(\Omega)$ in the middle: 
$$ 0\to \mathcal{E}\stackrel{\pi^*}\longrightarrow H^1(\Omega)\to\coker\pi^*\to 0$$
and
$$ 0\to \Inf(\Omega)\to H^1(\Omega)\stackrel{\tau}\longrightarrow
\freq(\Omega)\to 0.$$
This situation is considered, in the context of symbolic substitutions, in \cite{AndressRobinson}. 
We consider tilings that arise from substitutions, as well as tilings that arise from projection methods. For almost canonical projection tilings, both of these sequences split. We will give conditions under which this is true for substitution tilings, as well as examples in which one, or both, don't split.

In the next section, we briefly review the basics of substitution tilings, projection methods, tiling cohomology, and the construction of the maximal equicontinuous factor.
In Section \ref{The factor map and cohomology} we consider the map induced in cohomology by the maximal equicontinuous factor map, and in Sections \ref{Almost canonical projection patterns} and \ref{Substitution tilings} we restrict consideration to almost canonical projection patterns and substitution tilings, resp.
 
\section{Preliminaries}

\subsection{Tilings and their properties}
We will use the formulations and terminology of \cite{BargeKellendonk} and
just recall here what is necessary to set up the notation. 

An $n$-dimensional tiling is an infinite collection of tiles which cover
$\R^n$ and have pairwise disjoint interiors. Here a tile is a compact 
subset of $\R^n$ which is the closure of its interior. A tile may
carry a mark in case a distinction between geometrically
congruent tiles is necessary. 
A (finite) patch is a finite collection of tiles with pairwise disjoint
interiors. Its diameter is the diameter of the set covered by its
tiles.

The translation group $\R^n$ acts on tiles, patches and
tilings as on all geometric objects of $\R^n$ 
and we write this action by $t\cdot O$ or $O-t$ with
$t\in\R^n$ and $O$ the geometric object.
A collection $\Omega$ of tilings of $\R^n$
has (translationally) finite local complexity (FLC)
if it is the case that for each $R$ there are 
only finitely many translational equivalence classes of patches
$P\subset T\in\Omega$ with diameter smaller than $R$. A single tiling $T$
has FLC if $\{T\}$ has FLC. Finite local complexity of tilings will be
a standing assumption in this article and we won't repeat it.

We say that a collection $\Omega$ of tilings of $\R^n$ constitutes an 
  $n$-dimensional tiling space if $\Omega$ has FLC, is closed under
translation, and
is compact in the tiling metric $d$. In this metric two tilings are
close if a small translate of one agrees with the other in a large
neighborhood of the origin. The main example of a tiling space is the
hull of an FLC tiling $T$   
$$\Omega_T=\{T': T' \mbox{ is a tiling of $\R^n$ and every patch of
  $T'$ is a translate of a patch of $T$}\}.$$  
If the translation action on $\Omega$ is free (i.e., $T-v=T\Rightarrow
v=0$), $\Omega$ is said to be non-periodic.

Of particular interest for us are repetitive tilings which have the
property that for each finite patch $P$ of $T$ the set of occurrences of
translates of $P$ in $T$ is relatively dense. 
If $\Omega$ is repetitive, then the action of $\R^n$ on $\Omega$ by
translation is minimal. 

Another property which we will require occasionally is the existence of
frequencies of patches in a tiling. The frequency of a patch $P$
(up to translation) in $T$ is the density of the set of occurrences of
translates of $P$ in $T$, and being able to define this properly, independent of the limiting procedure, is equivalent to the
unique ergodicity of the dynamical system $(\Omega,\R^n)$. We denote,
then, the unique ergodic measure by $\mu$. 

Let $p$ be a puncture map; that is, $p$ assigns to each tile $\tau$ a point $p(\tau)\in
\tau$ so that $p(\tau+v)=p(\tau)$. 
If a tiling $T$ has FLC then 
the set of its punctures $p(T)=\{p(\tau):\tau\in T\}$ is a Delone set,
i.e.,\ a subset of $\R^n$ which is  uniformly discrete and
relatively dense. The puncture map $p$ defines a discrete hull 
$\Xi = \{T'\in\Omega_T:0\in p(T')\}$. $\Xi$ is also refered to as the
canonical transversal as it is transversal in $\Omega_T$ 
to the $\R^n$-action reducing it to the so-called tiling groupoid
$\G=\{(\omega,t)\in \Xi\times\R^n: \omega-t\in \Xi\}$ with
multiplication
$(\omega,t)(\omega',t') = (\omega,t+t')$ provided $\omega' = \omega-t$. 

The definitions we have made for tilings all have analogs for Delone
sets and whether we deal with tilings or Delone sets is mainly a
matter of convenience. One could, for instance,  
represent a tiling $T$ by the Delone set of its punctures, or a Delone
set by its Voronoi tiling and the
topological dynamical systems $(\Omega,\R^n)$ are unchanged. Whereas substitutions are usually
and more intuitively presented by tilings, the projection method 
produces Delone sets which are often referred to as projection
patterns (or, under more general circumstances, model sets).

\subsection{Substitution tilings}
Suppose that $\mathcal{A}=\{\rho_1,\ldots,\rho_k\}$ is a set of
translationally inequivalent tiles (called prototiles ) in $\R^n$ and $\Lambda$ is an expanding linear isomorphism of $\R^n$, that is, all eigenvalues of $\Lambda$ have modulus strictly greater than $1$. 
A substitution on $\mathcal{A}$ with expansion $\Lambda$ is a function $\Phi:\mathcal{A}\to\{P:P$ is a patch in $\R^n\}$
with the properties that, for each $i\in\{1,\ldots,k\}$, every tile in
$\Phi(\rho_i)$ is a translate of an element of $\mathcal{A}$, and
$\Phi(\rho_i)$ covers the same set as $\Lambda(\rho_i)$. Such a substitution
naturally extends to patches whose elements are translates of
prototiles by $\Phi(\{\rho_{i(j)}+v_j:j\in J\}):=\cup_{j\in
  J}(\Phi(\rho_{i(j)})+\Lambda v_j)$. A patch $P$ is allowed for
$\Phi$ if there is an $m\ge1$, an $i\in\{1,\ldots,k\}$, and a
$v\in\R^n$, with $P\subset \Phi^m(\rho_i)-v$. The substitution tiling
space associated with  
$\Phi$ is the collection $\OP:=\{T:T$ is a tiling of $\R^n$ and every
finite patch in $T$ is allowed for $\Phi\}$. Clearly, translation
preserves allowed patches, so $\R^n$ acts on $\OP$ by translation.   
We say that a substitution $\Phi$ is FLC or non-periodic if its corresponding 
tiling space $\Omega_\Phi$ is FLC or non-periodic.

The substitution $\Phi$ is primitive if for each pair $\rho_i,\rho_j$
of prototiles there is a $k\in\N$ so that a translate of $\rho_i$
occurs in $\Phi^k(\rho_j)$.  If $\Phi$ is primitive then $\OP$ is
repetitive. 

If $\Phi$ is primitive, non-periodic and FLC then $\OP$ is
compact in the tiling metric, $\Phi:\OP\to\OP$ is a homeomorphism, and
the translation action on $\OP$ is minimal and uniquely ergodic 
\cite{AP,S3}. In particular, $\OP=\Omega_T$ for
any $T\in\OP$. 
All substitutions will be assumed to be primitive, non-periodic and FLC.

\subsection{Almost canonical projection patterns}
We describe here almost canonical projection patterns without going
into details which the reader may find in \cite{FHK}.

Consider a regular lattice  $\Gamma\subset \R^n\times \R^{\np}$ such
that $\R^n$ is in irrational position w.r.t.\ $\Gamma$, and a window
$K$ which is a compact polyhedron.
Let $\pi:\R^n\times \R^\np\to \R^n$ be the projection onto the first factor and
$\pi^\perp:\R^n\times \R^\np\to \R^\np$ be the projection onto the
second factor. 
Define the set $S$ of {\em singular points} in $\R^\np$ by 
$$ S :=\bigcup_{\gamma\in\Gamma} \partial K - \pi^\perp(\gamma)$$
where  $\partial K$ denotes the boundary of $K$.
We assume that
\begin{itemize}
\item  the restrictions of $\pi^\|$ and $\pi^\perp$ to
$\Gamma$ are one to one,
\item  the restrictions of $\pi^\|$ and $\pi^\perp$ to
$\Gamma$ have dense image,
\item  there exists a finite set of affine hyperplanes $\{W_i\}_{i\in
    I}$ of codimension $1$ in $R^\np$ such that $S$ may be
  alternatively described as 
$$S = \bigcup_{i\in I}\bigcup_{\gamma\in\Gamma} W_i - \pi^\perp(\gamma).$$
\end{itemize}
We call the hyperplanes $W_i - \pi^\perp(\gamma)$, $i\in
I,\gamma\in\Gamma$ {\em cut-planes}.
By the second assumption $S$ is a dense subset of $\R^\np$ but of zero
Lebesgue measure. The last assumption means that, given a face $f$ of
$K$, the union of all $\pi^\perp(\Gamma)$-translates of $f$ contains the
affine hyperplane spanned by $f$; in particular 
the faces of $K$ have rational orientation w.r.t\ $\pi^\perp(\Gamma)$ and
the stabilizer $\{\gamma\in\Gamma:W_i-\pi^\perp(\gamma) = W_i\}$ of an affine hyperplane $W_i$ must have at least rank $\np-1$. 

We also assume (for simplicity) that $0$ is
not a singular point. Then the set
$$P_K:=\{\pi^\|(\gamma):\gamma\in\Gamma, \pi^\perp(\gamma)\in K\}$$ 
is a repetitive Delone set, called the projection pattern with window
$K$. With the above rather restrictive assumptions made on the window
$K$ the projection pattern is called {\em almost canonical}. There
are standard ways to turn $P_K$ into a tiling which is mutually
locally derivable from $P_K$,
for instance the dual of the Voronoi tiling defined by $P_K$ will do it.

\subsection{Tiling cohomology and the order structure on the top degree}\label{order
  structure}
  We are interested in the cohomology of a tiling (or pattern)
  $T$. This cohomology can be defined in various equivalent ways, for
  instance as the Cech cohomology $H(\Omega)$ of the hull $\Omega$ of
  $T$ or as (continuous cocycle) cohomology $H(\G)$ of the tiling
  groupoid $\G$ (after \cite{Renault}). 
The equivalence between the two formulations of tiling cohomology can either be seen by realizing that $\Omega$ is a classifying space for the groupoid, or by a further reduction:
From the work of Sadun-Williams \cite{SadunWilliams03}
we know that we can deform the tiling into a tiling by decorated cubes without changing the topological structure of the hull (the hull of the tiling by cubes is homeomorphic to the original one). It follows then that the tiling groupoid of the tiling by cubes is a transformation groupoid $\Xi'\times \Z^n$ which is continuously similar to $\G$ \cite{Renault,FHK}. Here $\Xi'$ is the canonical transversal of the tiling by cubes. Like $\Xi$ it is a compact totally disconnected space.
It follows then quickly that $H(\G)\cong H( \Xi'\times \Z^n)$ and, by definition of the groupoid cohomology, $H( \Xi'\times \Z^n)$  is the dynamical cohomology 
$H(\Z^n,C(\Xi',\Z))$ which is the cohomology of the group $\Z^n$ with
coefficients in the integer valued continuous (and hence locally
constant) functions. Now what the construction of \cite{SadunWilliams03} actually
does on the level of spaces is to realise $\Omega$ as a fiber bundle
over an $n$-torus whose typical fibre is $\Xi'$ such that the above $\Z^n$
action corresponds to the holonomy action induced by the fundamental group of
the torus. Another way of saying this is that $\Omega$ is the mapping
torus of that $\Z^n$ action.  
It follows (as is seen for instance from the Serre spectral sequence) that $H(\Z^n,C(\Xi',\Z))$
is isomorphic to $H(\Omega)$. 

In the highest non-vanishing degree, namely
in degree $n$, $H^n(\G)$ is the group of co-invariants, 
$$ H^n(\G)\cong C(\Xi,\Z)/B $$
where $B$ is the subgroup generated by differences of indicator functions of the
form $1_U-1_{U-t}$, $U\subset \Xi$ a clopen subset and $t\in\R^n$ such that $U-t\subset \Xi$.

The group of co-invariants carries a natural order: an element is positive whenever it is represented by a positive function in $C(\Xi,\Z) $.  Moreover, the order structure is preserved under groupoid isomorphism, and hence the ordered group of co-invariants is a topological invariant for the tiling system. 

Let us assume that the tiling system is strictly ergodic. Hence the $\R^n$ action on $\Omega$ as well as the groupoid action on $\Xi$ are minimal and uniquely ergodic. Let $\nu$ be the unique ergodic measure on $\Xi$. Then $1_U\mapsto \nu(U)$ factors through $C(\Xi,\Z)/B $ and hence, combined with the isomorphism
$ H^n(\Omega)\cong C(\Xi,\Z)/B $, defines a group homomorphism
$$\tau : H^n(\Omega) \to \R.$$
Now the order can be described by saying that  $x\in H^n(\Omega)$ is positive whenever $\tau(x)\geq 0$. We say that an element $x$ is infinitesimal if it is neither strictly positive nor strictly negative, which is hence the case if and only if $\tau(x)=0$. We denote the infinitesimal elements by $\Inf(\Omega)$.

It is well known that a basis of the topology of $\Xi$ is given by the acceptance domains on patches, that is, by subsets containing all tilings which have a given patch at the origin.
It follows from this (and the unique ergodicity) that $\nu(U_P) $ is the frequency of occurrence of the patch $P$ in $T$ where $U_P$ is the acceptance domain of $P$. Let us denote by   $\freq(\Omega)$ the subgroup of $\R$ generated by the frequencies of finite patches in $T$. 
We thus have an exact sequence 
\begin{equation}\label{exact sequence 1}
 0 \to \Inf(\Omega) \to H^n(\Omega)\stackrel{\tau}\to \freq(\Omega) \to 0 
 \end{equation}
which splits if $\freq(\Omega)$ is finitely generated.

\subsection{The maximal equicontinuous factor and eigenvalues}
Let $(X,G)$ be a minimal dynamical system with compact Hausdorff space $X$ and abelian group $G$ action. There is a maximal equicontinuous factor $(\Xmax,G)$ of that system -- unique up to conjugacy -- and this factor can be obtained from the continuous eigenvalues of the action. In fact,
a {\em continuous eigenfunction} of a dynamical system $(X,G)$ is a non-zero function
$f\in C(X)$ for which there exists a (continuous) character $\chi\in
\hat G$ such that  
\begin{equation}\nonumber
f(t\cdot x) = \chi(t) f(x).
\end{equation} 
$\chi$ is called the {\em eigenvalue} of $f$. To stress that this eigenvalue is an eigenvalue to a continuous function (as opposed to an $L^2$-function) one also calls it a {\em continuous} eigenvalue. But we will here consider only eigenvalues to continuous eigenfunctions and so drop that adjective.

The set of all eigenvalues
$\mathcal E$ forms a subgroup of the Pontryagin dual $\hat G$ of $G$.  
We consider $\mathcal E$ with discrete topology. Then the Pontryagin dual $\hat\E$ of $\E$ is a compact abelian group and
the maximal equicontinuous factor can be identified with it, $\Xmax\cong \hat\E$.
The factor map $\pi:X \to \hat{\mathcal E}$ is then given by $x\mapsto j_x$, where $j_x:\mathcal E\to \T^1$ is defined by $j_x(\chi) = f_\chi(x)$, and the $G$-action on 
$\varphi\in\hat{\mathcal E}$ is given by
 $(t\cdot \varphi)(\chi) = \chi(t)\varphi(\chi)$.  Here  
 $f_\chi$ is the eigenfunction to eigenvalue $\chi$ normalized in such a way that $f_\chi(x_0)=1$ 
 where  $x_0\in X$ is some chosen point used to normalize all eigenfunctions.



\section{The factor map and cohomology}\label{The factor map and cohomology}

We are interested in the map in cohomology induced by the factor map $\pi$:
$$ \pi^*:H^k(\hat\E)\to H^k(X).$$
(If nothing else is said this means integer valued \v{C}ech
cohomology.) 
In particular, we consider the kernel and cokernel of $\pi^*$.
The situation is extremely simple in degree $0$: 
$X$ and $\Xmax$ are connected and so their cohomology in degree $0$ is
$\Z$ and $\pi^*$ an isomorphism in that degree. The situation is very
complicated in degrees larger than one, and we will only be able
to say something for almost canonical projection patterns. This will
be done in the next section.   
In this section we will concentrate on degree $1$ which is important for
deformation theory
\cite{SadunWilliams03,ClarkSadun06,Kellendonk08,Boulmezaoud10} and for
the homological version of the Pisot conjecture \cite{BBJS, BG}. 

\subsection{The cohomology of the maximal equicontinuous factor}
Note that the group $\E$ of eigenvalues is at most countable. This follows from the fact that $L^2(X,\mu)$ is separable (for any ergodic invariant probability measure $\mu$) and eigenfunctions to distinct eigenvalues are orthogonal in that Hilbert space.  

We suppose that $\E$ is torsion free which is certainly the case if $\hat G$ is torsion free, in particular thus if $G=\R^n$. 
As an abelian group $\E$ is a $\Z$-module and we may consider the exterior algebra $\Lambda\E$, which is a graded ring. 

As is well-known, $H^1(S^1)$ is a free abelian group of rank one. We pick a generator $\gamma\in H^1(S^1)$ (which amounts to choosing an orientation). Given an element of $\chi\in\E$, which we may view as a character on $\hat\E$, $\chi:\hat\E\to S^1$,
$\chi^*(\gamma)$ defines an element in $H^1(\hat \E)$ and thus a group homomorphism $\jmath:\E\to H^1(\hat \E)$, $\jmath(\chi) = \chi^*(\gamma)$. 
\begin{theorem} 
$\Lambda\jmath : \Lambda\E\to H(\hat\E)$ is a graded ring isomorphism.
\end{theorem}
\begin{proof}
As $\E$ is countable and torsion free we can write it as $\E = \dlim(\E_n,i_n^{n+1})$ where $\E_n$ is free abelian of finite rank and $i_n^{n+1}:\E_n\to \E_{n+1}$ an injective group homomorphism\footnote{$\E = \{g_n,n\in\N\}$ and we may take $\E_n$ to be the group generated by $\{g_1,\cdots,g_n\}$.}. We denote $i_n:\E_n\to \E$ the corresponding group inclusion. It follows that
$\hat\E = \ilim (\hat\E_n,\hat i_n^{n+1})$. Now $\hat\E_n$ is a torus whose dimension equals the rank of $\E_n$ and so its cohomology is generated as a ring by its degree $1$ elements which, in turn, are the elements of the form $\jmath(\chi)$, 
$\chi \in \E_n$. This shows that 
$H(\hat\E_n) \cong \Lambda\E_n$ with ring-isomorphism given by $\Lambda\jmath: \Lambda\E_n\to H(\hat\E_n) $. Hence
$H(\hat\E) = \dlim H(\E_n,{i_n^{n+1}}^*) \cong \dlim (\Lambda\E_n, \Lambda i_n^{n+1}) = \Lambda\E$.
\end{proof}

\subsection{Injectivity of $\pi^*$ in degree one.}

Let $[X,S^1]$ denote the set of homotopy classes of maps from $X$ to the circle $S^1:=\{z\in\mathbb C:|z|=1\}$. This is an abelian group (known as the {\em Bruschlinsky group} of $X$) under the operation $[f]+[g]:=[fg]$ and the map $[f]\mapsto f^*(\gamma)$ is a natural isomorphism between $[X,S^1]$ and ${H}^1(X)$ (see, for example, \cite{PT}). 
If we take $X = \hat\E$, the maximal equicontinuous factor, then $H^1(\hat\E) = \E$ and the isomorphism 
$\E \cong [\hat \E,S^1]$ is given by $\chi \mapsto [\chi]$, as is easily seen.
The naturality of the isomorphism implies that 
\[ \begin{array}{ccc}
H^1(\hat\E)&\stackrel{\pi^*}\longrightarrow & H^1(X) \\
\| & {} & \| \\
{}[\hat\E,S^1] 
& \stackrel{\pi^*}\longrightarrow & [X,S^1]
\end{array}\]
commutes.

We now suppose that $G=\R^n$ and $X=\Omega$ is a tiling space. Then $\hat{\R^n}$ is isomorphic to ${\R^n}^*$, the dual of $\R^n$ as a vector space, the map 
${\R^n}^*\ni \beta \mapsto e^{2\pi\imath\beta} \in \hat{\R^n}$ providing a group isomorphism. We denote $E = \{\beta:e^{2\pi\imath\beta} \in \E\}$, calling it also the group of eigenvalues,
let $\imath:E\hookrightarrow {\R^n}^*$ be the inclusion, and $\theta:E\to [\Omega,S^1]$ be the composition $\theta(\beta) = \pi^*([e^{2\pi\imath\beta}])$. If $f_\beta$ is an eigenfunction to eigenvalue $\beta$, normalized so that its modulus is everywhere $1$, 
then $\theta(\beta) = [f_\beta]$. Indeed, by minimality any two
eigenfunctions differ by a multiplicative constant and hence are
homotopic.

The Lie-algebra of $G=\R^n$ is $\R^n$. 
Let $H^k(\R^n,C^\infty(\Omega,\R))$ be the Lie-algebra cohomology of $\R^n=Lie(G)$ with values in 
$C^\infty(\Omega,\R)$ which are continuous functions that are smooth w.r.t.\ the derivative $d$ defined by the Lie-algebra action. Since any continuous function on $\Omega$ can be approximated in the sup norm by a smooth function we can define a group homomorphism
$$\psi :[\Omega,S^1] \to H^1(\R^n,C^\infty(\Omega,\R)),\quad  \psi([f]) = \frac{1}{2\pi\imath} f^{-1}df $$
using a smooth representative. 
Given an ergodic invariant probability measure $\mu$ on $\Omega$ we can define the homomorphism
$$ C_\mu:H^1(\R^n,C^\infty(\Omega,\R))\to {\R^n}^*,\quad C_\mu( \alpha) = \int_\Omega \alpha(\omega) d\mu(\omega).$$
The composition $\tau=\C_\mu\circ \psi$ is the degree $1$ part of the Ruelle-Sullivan map of \cite{KellendonkPutnam06}. $\tau([f])(v)$ is also known as the Schwartzman
winding number of $f$ with respect to the $\R$-action $T\mapsto T-tv$ on $\Omega$.
\begin{lemma}\label{lem-2}
$\tau\circ\theta:E\to {\R^n}^*$ is given by $\tau\circ\theta=\imath$.
\end{lemma}
\begin{proof} Since $f_\beta^{-1}df_\beta = 2\pi\imath \beta$, a constant function, we have 
$C_\mu(\Psi(\theta(\beta)))=\beta$.
\end{proof}
\begin{cor}
$\pi^*$ is injective in degree $1$.
\end{cor}
\begin{proof}
$\tau\circ\theta$ is injective and factors through the degree one part of $\pi^*$.
\end{proof}
\noindent
{\bf Remark:} Lemma~\ref{lem-2}
is actually the degree one part of a more general result which can be obtained with the help of the full Ruelle-Sullivan map $\tau: H(\Omega) \to \Lambda{\R^n}^*$. The composition
$\tau\circ\pi^*\circ \Lambda\jmath : \Lambda E \to \Lambda {\R^n}^*$ is $\Lambda\imath$. In particular $\mbox{\rm im}\tau\circ\pi^*=\Lambda\imath(E)$.
The proof of this statements is a slight generalization of the one
given in \cite{KellendonkPutnam06}[Thm.~13]. In degree $n$ we may
identify $\Lambda^n{\R^n}^*\cong\R$ and obtain the same map as above
(justifying this way the double use of $\tau$ in the notation).
\medskip


\subsection{The cokernel of $\pi^*$ in degree one.} We start with the remark that 
$H^1(\Omega)$ is torsion free, as follows from the universal coefficient theorem. 
In higher degrees, $H^k(\Omega)$ may contain torsion.
\renewcommand{\T}{S^1}
\begin{lemma}\label{Kras}(Krasinkiewicz \cite{Kr})
Suppose that $f : \Omega \to \T$ is continuous and suppose that there are $0\ne k\in\Z $ and a continuous
$g:\Omega\to\T$ such that $[f]=k[g]\in[\Omega,\T]$. Then there is a continuous $\tilde{f}:\Omega\to\T$ so that $f = \tilde f^k$.
\end{lemma}

\begin{proof} Let $p_k:\T\to\T$ be the $k$-fold covering map $p_k(z):=z^k$.
By assumption we have $[f]=[p_k\circ g]$. Let $H:\Omega\times I\to\T$ be a homotopy from
$h:=p_k\circ g$ to $f$. 
Then $\tilde{h}:= g$ is a lift of $h$. Being a covering map, $p_k$ has the homotopy lifting property so there is a homotopy $\tilde{H}:\Omega\times I\to\T$ from $\tilde{h}$ to some function $\tilde{f}$ such that $\tilde{H}(\cdot,0)=\tilde{h}$ and $p_k\circ\tilde{H}=H$. It follows that $p_k\circ\tilde{f}=f$.\end{proof}

\begin{Theorem}\label{coker torsion free}
The cokernel of the homomorphism $\pi^*:H^1(\hat\E)\to{H}^1(\Omega)$ is torsion free.
\end{Theorem}

\begin{proof}
The statement of the theorem is equivalent to saying that the
cokernel of $\theta:E\to{H}^1(\Omega)$ is torsion free.
Let $[f]\in[\Omega,\T]\cong{H}^1(\Omega)$, $k\in\N$, and $\beta\in E$ be such that $k[f]=\theta(\beta)=[f_{\beta}]$. By Lemma \ref{Kras} there is an $\tilde{f}_{\beta}:\Omega\to\T$ so that $f_{\beta}=p_k\circ\tilde{f}_{\beta}$. Then $k([\tilde{f}_{\beta}]-[f])=0$.
Since ${H}^1(\Omega)$ is torsion free, $[\tilde{f}_{\beta}]=[f]$. We claim that
$\tilde{f}_{\beta}$ is an eigenfunction with eigenvalue $\beta/k$.

By continuity of $\tilde f_\beta$ it is enough to verify the equation $\tilde{f}_{\beta}(T_0-x) = \exp{2\pi\imath\frac{\beta(x)}k} \tilde f_\beta(T_0)$ for some $T_0$.
We have $f_{\beta}(T_0-x)=\exp(2\pi\imath\beta(x))f_\beta(T_0)$ for all $x\in\R^n$, thus
$\tilde{f}_{\beta}^k(T_0-x)= \exp(2\pi\imath\beta(x))\tilde f_\beta^k(T_0)$. Taking the $k$th root we obtain $\tilde{f}_{\beta}(T_0-x)= u(x) \exp(2\pi\imath\frac{\beta(x)}k)\tilde f_\beta(T_0)$ where $u(x)$ is a $k$th  root of unity. Continuity of
$\tilde{f}_{\beta}$ requires that $u(x)=1$. Hence $\tilde{f}_{\beta}$ is an eigenfunction with eigenvalue $\beta/k$. Hence $\theta(\beta/k)=[\tilde{f}_{\beta}]=[f]$ and $\coker \theta$ is torsion free.
\end{proof}
\begin{cor}\label{cor-6} If $\coker\pi^*$ is finitely generated then $H^1(\Omega)$
is isomorphic to the direct sum of  $E$ with  $\coker\pi^*$.
\end{cor}
\begin{proof} Under the assumption $\coker\pi^*\cong \Z^l$ for some finite $l$, as it is torsion free.  
\end{proof}
If $H^1(\Omega)$ is not finitely generated then it is not always  the direct sum of  $E$ with  $\coker\pi^*$ as Example~\ref{ex-3l} shows.

\section{Almost canonical projection patterns}\label{Almost canonical projection patterns}
To obtain a Cantor fiber bundle for almost canonical projection
tilings we do not actually use the approach of Sadun \& Williams via a deformation of
the tiling, but rather consider a variant of the ``rope dynamical
system'' of \cite{Kellendonk}, 
see \cite{FHK}. 
This way we obtain a different Cantor fibre bundle, whose fibre we
shall denote by $\Fi$. The conclusion that tiling cohomology can
be formulated as group cohomology of a $\Z^n$-action on $\Fi$ remains
valid and we have the benefit that the structure of $\Fi$ allows for a
calculation of the cohomology groups. $\Fi$ can be obtained
from the set of singular points $S$ 
by disconnecting $\R^\np$ along the cut-planes
$W_i-\pi^\perp(\gamma)$ 
and moding out the action of a subgroup of
$\pi^\perp(\Gamma)$ of rank $\np$. This subgroup $\Z^\np$ should be a
direct summand, i.e.\ 
$\pi^\perp(\Gamma)=\Z^\np\oplus\Z^n$, and it should span $\R^\np$,
but it can otherwise be chosen arbitrarily. So $\Fi=F_c/\Z^\np$, where
$F_c$ is the so-called cut-up  space obtained by disconnecting
$\R^\np$, and the other summand $\Z^n$ yields the action on $\Fi$.   
We  refer the reader to \cite{FHK,GHK} for the precise definition of
the disconnecting procedure mentioning here only that it can be
obtained via an inverse limit:  For any finite collection of
cut-planes, disconnecting $\R^\np$ along these cut-planes means taking
out the cut-planes so that the remaining part of $\R^\np$ falls into
several connected components and then completing separately these
connected components to obtain a closed space. The inverse limit is
just geared to make that work for infinitely many cut-planes. 

To do the actual computation it is more convenient to work with homology. Using Poincar\'e
duality for group (co-) homology and the fact that $\Z^\np$ acts
freely on $F_c$ one obtains
$$ H^k(\Z^n,C(\Fi,\Z))\cong H_{n-k}(\Z^n,C(\Fi,\Z))\cong
H_{n-k}(\Gamma,C_\np)$$
where $C_\np$ is the $\Z$-module generated by indicator functions on
polyhedra whose faces belong to cut-planes and
$\gamma\in\Gamma$ acts 
on such a function by pull back 
of the translation with $\pi^\perp(\gamma)$. Intersections of
cut-planes are affine subspaces of smaller dimension and we call such
an affine space a singular space. On each singular subspace $L$, say
of dimension $k$, we
have a similar structure as on $\R^\np$: The
intersections of the cut-planes with $L$ are affine subspaces of
co-dimension $1$ in $L$. We let $C_k$ be the module generated by
indicator functions on $k$-dimensional polyhedra in a 
$k$-dimensional singular space
whose faces belong to cut-planes.  

The polyhedral structure
and the fact that $\R^\np$ is contractible give rise to an 
acyclic complex $C_\np\to C_{\np-1}\to \cdots\to C_0$ of
$\Gamma$-modules whose
differential is reminiscent of the boundary map in polyhedral
complexes. 
As a result the homology $H_*(\Gamma,C_\np)$ may be
computed by breaking the complex into $\np$ short exact sequences
$$ 0\to C_k^0 \to C_k\to  C_{k-1}^0\to 0,\quad 0\leq k < \np$$
with $C_k^0$ equal to the image of the boundary map $\delta:
C_{k+1}\to C_k$ which is, of course, the same as the kernel of $\delta:
C_{k}\to C_{k-1}$
(and thus $C_\np = C_{\np-1}^0$) and $C_{-1}^0=\Z$.
Each such short exact sequence gives rise to a long exact sequence in
homology and in particular to a connecting homomorphism
$\gamma_k:H_p(\Gamma,C_{k-1}^0)\to H_{p-1}(\Gamma,C_{k}^0)$. 
We now recall:
\begin{theorem}[\cite{BargeKellendonk}]
The maximal equicontinuous factor $\hat \E$ of a projection pattern is
naturally isomorphic to the torus $\R^n\times\R^\np/\Gamma$.
\end{theorem}
Hence, upon identifying $\R^n\times\R^\np$ with its dual we have
$E = \Gamma$ and so we may identify $H^p(\hat \E) = \Lambda^p\Gamma =
H^p(\Gamma,\Z)\cong H_{n+\np-p}(\Gamma,\Z)$, the last identification
by Poincar\'e duality. 
\begin{theorem}[\cite{GHK}]
Under the above identifications $H^p(\Omega)\cong
H_{n-p}(\Gamma,C_{\np})$ and $H^p(\hat \E) \cong H_{n+\np-p}(\Gamma,\Z)$
the map $\pi^*:H^p(\hat \E)\to H^p(\Omega)$ gets identified with the
composition of 
connecting maps $\gamma_{\np-1}\circ\cdots\circ \gamma_0:H_{n+\np-p}(\Gamma,\Z)\to H_{n-p}(\Gamma,C_{\np})$.
\end{theorem} 
\subsection{Injectivity of $\pi^*$.}
We have now the tools at hand to find out in which degrees $\pi^*$ is
injective. In fact, the long exact sequence in homology corresponding
to the above exact sequence is
$$ \to H_p(\Gamma,C_k)\stackrel{\delta_*'}{\to} H_p(\Gamma,C_{k-1}^{0})\stackrel{\gamma_k}{\to} H_{p-1}(\Gamma,C^0_k){\to} $$
where $\delta'$ is the boundary map with target space restricted to
its image. Thus $\gamma_k$ is injective whenever
$\delta_*'=0$. Now the module $C_k$ decomposes, $C_k =
\bigoplus_{\theta\in I_k} C_k^\theta \otimes \Z[\Gamma/\Gamma^\theta]$
where $I_k$ indexes the set of $\Gamma$-orbits of singular spaces of
dimension $k$ and $\Gamma^\theta\subset \Gamma$ is the subgroup
stabilizing the singular space of orbit type $\theta$. It follows 
that  $H_p(\Gamma,C_k)=\bigoplus_{\theta\in I_k}
H_p(\Gamma^\theta,C_k^\theta)$.   
Since the singular spaces which make up $C_k^\theta$ are $k$-dimensional
$H_p(\Gamma^\theta,C_k^\theta)=0$ if $p>k$. But we have also $H_p(\Gamma^\theta,C_k^\theta)=0$ if $p>\rk\Gamma^\theta - \dim \R \Gamma^\theta$, because $\Gamma^\theta$ contains a subgroup of rank $\dim \R \Gamma^\theta$ which acts freely on $C_k^\theta$. To summarize
$$H_p(\Gamma,C_k)=0 \quad\mbox{if}\quad p > \min\{k,r_k\}$$
where $r_k = \max_{\theta\in I_k} (\rk\Gamma^\theta - \dim \R \Gamma^\theta)$.
In particular,
$\gamma_{\np-1}:H_{p+1}(\Gamma,C_{\np-2}^0)\to H_{p}(\Gamma,C_{\np})\cong H^{n-p}(\Omega)$ is injective if $p\geq \min\{\np-1,r_{\np-1}\}$.  

We now consider first the case in which
the ranks of the stabilizers are minimal. Since the stabilizer
of  $W_i$ must have at least rank $\np-1$ the minimal case is 
$r_{\np-1} = 0$ which then implies that $r_k= 0$ for all $k$. 
This is in fact the generic case and it corresponds to the pattern
having maximal complexity among almost canonical projection patterns;
that is, the growth exponent for the complexity function is $\np
n$ \cite{Julien}. We see from the above that $H_p(\Gamma,C_k^\theta)=0$ if $p>0$ and
therefore $\pi^*$ is injective in all degrees. But 
$H^k(\Omega)$ is infinitely generated except if $\np=1$.

The situation is different if we require that
the cohomology is finitely generated. By the results of \cite{FHK} and
\cite{Julien} this is precisely the case if $\nu:=\frac{n+\np}\np$ is
an integer and the rank of the stabilizer of a singular plane is $\nu$
times its dimension, i.e.\ $r_k = (\nu-1)k\geq k$. We find it interesting to note that this case corresponds to the case of minimal complexity, that is, the number of patches of size $R$ grows polynomially with exponent $n$ \cite{Julien}.
This yields the bound that $\pi^*:H^k(\hat \E)\to H^k(\Omega)$ is
injective if $k\leq n - \np + 1=(\nu-2)n + 1$. 
Furthermore, the calculations done in \cite{FHK} (for codimension $3$ patterns) show that this is the best possible bound: $\pi^*$ is never injective in degree $k > (\nu-2)n + 1$. 
In particular, for the standard tilings like the Penrose, Amman-Beenker, Socolar, and the icosahedral tilings, $\nu=2$, and hence $\pi^*$ is injective only in degree $0$ and $1$.
\subsection{On the cokernel of $\pi^*$.} We consider here only the case of
finitely generated cohomology. It comes not as a surprise that then
$\coker\pi^*$ is also finitely generated, see \cite{FHK}. It can,
however, have torsion in higher degrees \cite{GHK-ICQ,GHK}: the T\"ubingen
Triangle Tiling is an example of a $2$-dimensional tiling which has
torsion in its second cohomology and all icosahedral
tilings considered in \cite{GHK} have torsion in degree $3$ (the
Ammann-Kramer and the dual canonical $D_6$ tiling also have torsion in degree
$2$). A more subtle question is whether $\pi^*$ is always onto a
direct summand; that is, whether the exact sequence 
\begin{equation}\label{exact sequence 2}
0\to H^k(\hat\E)\stackrel{\pi^*}\longrightarrow H^k(\Omega) \to \coker\pi^* \to 0.
 \end{equation}
splits and hence the torsion in the tiling cohomology agrees with the
torsion of $\coker\pi^*$.
While this is generally true if $\np\leq 2$ 
it could be answered affirmatively in higher codimension $\np$ only
for rational
projection patterns \cite{GHK}. 
Rational projection patterns are not only
almost canonical but satisfy an additional rationality assumption,
roughly that the cut-planes are projections onto $\R^\np$ of lattice
planes in $\Q\Gamma$. This assumption  
allows for the construction of a torus arrangement in $\hat \E$. Such
torus arrangements were was first proposed  in \cite{Kalugin}.  

\subsection{The frequency module}
In this section we will prove that the frequency module,
$\freq(\Omega)$, of an almost canonical projection pattern is always
finitely generated and hence the sequence (\ref{exact sequence 2}) splits.

We start with some known background material. It is known that the
factor map $\pi$ is almost everywhere one-to-one and the measure on
$\Omega$ the push forward of the (normalized) Haar-measure on
$\hat\E$. This implies that the frequency module is generated by the
volumes of all polyhedra in $\R^\np$ whose faces lie in $S$. Here the
volume of a polyhedron is measured with the help of the Lebesgue measure
normalized so that the window $K$ has volume $1$.

\begin{theorem}
The frequency module
$\freq(\Omega)$ of an almost canonical projection pattern is finitely generated.
\end{theorem}
\begin{proof}
We call a point $x\in\R^\np$ a cut-point if it is the unique point in
the intersection of $\np$ cut-planes. Clearly, any polyhedron whose
faces lie in $S$ has vertices which are cut-points. Given that the
cut-points are dense we may subdivide any such polyhedron into
simplices of dimension $\np$ such that all vertices of the simplices
are cut-points (we do not care whether the newly introduced faces
lie in $S$). The theorem thus follows if we can show that the
$\Z$-module of volumes of all $\np$-simplices whose vertices are
cut-points is finitely generated.

Denote by $\cP$ the cut points. We claim that $\cP-\cP$ is contained in a
finitely generated $\Z$-module. Suppose first that $x,y\in\cP$ lie in a
common singular space $L$ of dimension $1$
and both on the intersection with the same class of
cut-plane, i.e.\ $\{x\} = L\cap (W_i-\pi^\perp(\gamma_x))$ and $\{y\} =
L\cap (W_i-\pi^\perp(\gamma_y))$. Then $x-y$ belongs to
$\pi^L_{W_i}(\Gamma)$, the projection along $W_i$ onto $L$ of
$\pi^\perp(\Gamma)$. 
This group is, of course, finitely generated and since there
are only finitely many $W_i$ we see that   $(\cP\cap L) -(\cP\cap L)$
is contained in a finitely generated $\Z$-module. Now we can go from
any
cut point $x$ to any other cut point $y$ along singular lines and since
there are only finitely many directions of singular lines $x-y$ lies in a
finitely generated $\Z$-module $M$, say.

The volume of an $\np$-simplex with vertices $(x_0,\cdots,x_\np)$ is
one half of the determinant of the $\np$ vectors
$x_i-x_0$, $i=1,\cdots,\np$ which all lie in $M$. Hence the determinant also lies in a finitely generated $\Z$-module.
\end{proof}
\begin{cor}For almost canonical projection patterns the sequence
  (\ref{exact sequence 2}) splits.
\end{cor}

\section{Substitution tilings}\label{Substitution tilings}
We recall quickly how to calculate the cohomology of substitution tiling space $\Omega=\Omega_\Phi$ referring the reader to \cite{AP} for more details.

The collared Anderson-Putnam complex $Y$ is an $n$-dimensional
CW-complex whose $n$-cells are collared prototiles. Two of these cells
are glued along $(n-1)$-faces if translates of the corresponding
collared prototiles meet along a translate of that face in some tiling
in $\Omega$. There is a natural map $p:\Omega\to Y$  
assigning to a tiling the point in $Y$ which corresponds to the position of the origin $0$ in the collared tile of the tiling that contains $0$. Furthermore, the substitution $\Phi$ induces a continuous surjection $F:Y\to Y$ with $p\circ\Phi=F\circ p$. By the universality property of the inverse limit,
$p$ induces a map $\hat{p}:\Omega\to \inv (Y,F)$ where $\inv (Y,F)$
denotes the inverse limit of the stationary system $\cdots
Y\stackrel{F}{\to} Y  \stackrel{F}{\to} Y$. It is shown in \cite{AP}
that $\hat p$ is a homeomorphism that conjugates $\Phi$ with the shift
$\hat{F}$ on $\inv (Y,F)$ and hence $\hat{p}:H^k(\inv (Y,F))\to
H^k(\Omega)$ is an isomorphism. By the continuity property of \v{C}ech
cohomology, $H^k(\inv (Y,F))$ is naturally isomorphic with $\dlim
(H^k(Y),F^*)$. This direct limit is computable and we are interested
in $k=1$.

$H^1(Y)$ is a free abelian group of finite rank, and so the stationary system  $H^1(Y)\stackrel{F^*}{\to} H^1(Y)  \stackrel{F^*}{\to} H^1(Y)\cdots $ is of the form $\Z^N\stackrel{A}{\to} \Z^N  \stackrel{A}{\to} \Z^N\cdots$ for some $N$ and $N\times N$ integer matrix $A$. 
Let $ER(A)=\bigcap_n A^n\Q^N$ be the eventual range of $A$. $ER(A)
=\Q^N$ if $A$ has non-vanishing determinant, but always $ER(A)=A^N\Q^N$. Then
$$\dlim (\Z^N,A) =\{v\in ER(A) : \exists n A^n v\in\Z^N\}= \bigcup_n \tilde A^{-n}\Sigma,$$
where $\Sigma = ER(A)\cap \Z^N$ and $\tilde A$ is the restriction of $A$ to $ER(A)$.

If the substitution forces its border then the above construction
works already if one considers non-collared tiles \cite{AP} for the
construction of the Anderson-Putnam complex $Y$. 
In the one dimensional context, that is for a substitution which can
be symbolically defined on an alphabet of $N$ letters, and for a
substitution which forces its border in the sense 
that all substituted tiles start with the same tile and all end with the same tile 
(that is, the symbolic substitution has a common prefix and a common suffix), one may replace
the Anderson-Putnam complex $Y$ simply by a bouquet, $X$, of circles, one circle for each letter. In this case, $H^1(X) \cong \Z^N$ and the matrix $A$ representing $H^1(X)  \stackrel{F^*}{\to} H^1(X)$ in the basis provided by the cohomology classes of the circles is the transpose
of the incidence matrix for the substitution. It turns out that, at least for
determining the cohomology,  
we may also work with the bouquet $X$ as long as the symbolic
substitution has either a common prefix or a common suffix
\cite{AndressRobinson,bd1}.

\subsection{One dimensional irreducible substitutions}
For one-dimensional tilings the first cohomology group arrises in both
exact sequences,  the degree $1$ version of the sequence (\ref{exact sequence 2}),
namely 
\begin{equation}\label{exact sequence 3}
 0\to E\stackrel{\theta}\longrightarrow H^1(\Omega)\to \coker\theta\to 0
 \end{equation}
and, assuming unique ergodicity, sequence (\ref{exact sequence 1})
$$ 0\to \Inf(\Omega)\to H^1(\Omega)\stackrel{\tau}\longrightarrow
\freq(\Omega)\to 0.$$
Some of the main results of \cite{AndressRobinson} give complete
information about the structure of the above sequences
in the context of one-dimensional primitive, irreducible substitutions
of FLC. The work
in \cite{AndressRobinson} starts with symbolic substitutions which are then
realized geometrically by assigning a length to each symbol so as to
realize it as an interval. Our case is slightly more restrictive, and
can be compared if the lengths of the symbols are obtained from the left
Perron Frobenius vector of the substitution matrix. The results of 
\cite{AndressRobinson} require generally that the substitution forces
its border on one side in the sense that the symbolic
substitution has a common prefix. A further assumption made on the
substitution is that the characteristic polynomial of its substitution
matrix is irreducible or, what amounts to the same, the dilation
factor $\lambda$ is an algebraic integer of degree equal to the number
of prototiles (letters). One simply says that the substitution is
irreducible in that case. 
\begin{theorem}[\cite{AndressRobinson}]
Consider a one-dimensional substitution tiling with common
prefix. Assume furthermore that the substitution is irreducible. 
Then $\Inf(\Omega)=0$. 
Moreover, if its dilation factor is a Pisot number then $\coker\theta = 0$.
\end{theorem}
Note that combined with Solomyak's result (\cite{S3}) on the existence of eigenfunctions, this yields a dichotomy:
Either $\lambda$ is a Pisot number and then $\coker\theta = 0$ or 
$\lambda$ is not a Pisot number and then $E = 0$.

Hence we see that in one dimension, under the assumptions of common
prefix and of irreducibility, the two exact sequences
(\ref{exact sequence 1}) and (\ref{exact sequence 3}) are completely
degenerate. We will see below in the examples that the situation is not at all
like this if we look at non irreducible substitutions. Also, the results
on projection patterns indicate that this behavior is restricted to
one-dimensional tilings.
 
\subsection{On the splitting of exact sequence (\ref{exact sequence 3})}
For substitution tilings, there is a simple criterion guaranteeing
that the sequence (\ref{exact sequence 3})
splits.  Clearly, if $F^*$ is an isomorphism in degree $1$ then
$H^1(\Omega) = H^1(Y)$ and since the latter is finitely generated we obtain from 
Cor.~\ref{cor-6} that the sequence (\ref{exact sequence 3}) splits. But we can do better.

Recall that the homeomorphism $\Phi:\Omega\to \Omega$ defined by a substitution on the substitution
tiling space satisfies
$$ \Phi(T - v) = \Phi(T)-\Lambda(v).$$
Let $f$ be an eigenfunction with eigenvalue $\beta\in E$. Then
$$ f(\Phi(T-v)) = e^{2\pi\imath \beta(\Lambda(v))} f(\Phi(T)),$$
showing that $f\circ\Phi$ is an eigenfunction with eigenvalue
$\Lambda^T\beta$. It follows that $\Phi^*\theta(\beta) = [f\circ\Phi]
= \theta(\Lambda^T\beta)$, i.e.\ $\theta$ intertwines the action of
$\Lambda^T$ on $E$ with that of $\Phi^*$ on $H^1(\Omega)$. $\Phi^*$
thus 
induces a homomorphism $\bar\Phi^*$ on $\coker\theta$.
We now work with rational coefficients, i.e.\ rational cohomology. Then, since $H^1(Y;\Q)$ is a finite dimensional vector space, also $H^1(\Omega;\Q)$ and thus $\coker_\Q\theta= H^1(\Omega;\Q)/\theta(E\otimes_\Z\Q)$ are finite dimensional and 
so we can view $\bar\Phi^*$ as a finite matrix with rational coefficients which we denote by $\bar A$. Of course, this matrix depends on a choice of basis, but not its determinant.

\begin{prop}\label{split cohomology}
Suppose that $\bar{A}$ is as above and that $\det(\bar{A})=\pm 1$. Then the sequence (\ref{exact sequence 3}) splits with
${H}^1(\Omega)\cong E\oplus\Z^l$, $l=\dim\coker_\Q\theta$.
\end{prop}
\begin{proof} Recall from above that we may identify  
$H^1(Y)=\Z^N$ and $F^*=A$ so that $H^1(\Omega) \cong \{v\in ER(A):\exists n A^n v\in\Sigma\}$ where $\Sigma = \Z^N\cap ER(A)$. We now denote  $ER_\Z:= \{v\in ER(A):\exists n A^n v\in\Sigma\}$ and by
$V_\Z\subset ER_\Z$ the subgroup corresponding to
$\theta(E)$ under the above isomorphism. Then (\ref{exact sequence 3}) can be identified with 
\begin{equation}\label{eq-split} 0 \to V_\Z \hookrightarrow ER_\Z \to ER_\Z/V_\Z \to 0
\end{equation}
and our aim is to show that there is a splitting map $s:ER_\Z \to V_\Z$.

Note that $H^1(\Omega;\Q) \cong ER(A)$ and let denote $V$  the subspace corresponding to
$\theta(E\otimes_\Z \Q)$ under this isomorphism. $V$ is the rational
span of $V_\Z$. 
Then 
$\bar{A}$ can be seen as the linear map induced on $W=ER(A)/V$ by $A$.
Let $\pi_W$ denote the natural projection of $ER(A)$ onto $W$ and let $\Gamma:=\pi_W(\Sigma)$. Then $\Sigma$ and $\Gamma$ are forward invariant under $A$ and $\bar{A}$, resp., and it follows from $\det(\bar{A})=\pm1$ that $\bar{A}$ restricts to an isomorphism of $\Gamma$. Working with rational vector spaces, the corresponding exact sequence
$0 \to V \to ER(A) \to W \to 0$ splits and there is a linear map $s':W\to ER(A)$ such that $\pi_W\circ s' = \mbox{\rm id}$. Let $\pi_V:ER(A) \to V$ be the projection onto $V$ with kernel $s'(W)$. We claim that $\pi_V(ER_\Z) = V_\Z$ which then shows that the restriction of $\pi_V$ to $ER_\Z$ is a splitting map $s$ for the sequence (\ref{eq-split}).

By Theorem~\ref{coker torsion free}, $ER_\Z/V_\Z$ is torsion free and hence $V_\Z = V\cap ER_\Z$. Indeed, if $x\in V\cap ER_\Z$, there is $p$ such that $px\in V_\Z$ and thus, if $x \notin V_\Z$, it would map to a $p$-torsion element in the quotient $ER_\Z/V_\Z$. So we need only to show that $\pi_V(ER_\Z) \subset ER_\Z$.

Let $x \in ER_\Z$, i.e.\ $x\in ER(A)$ and there is $n$ such that $A^nx\in\Sigma$. We have $A^nx = v + \gamma$ with $v=\pi_V(A^nx)\in V$ and $\gamma = s'\circ \pi_W(A^nx) \in s'(\Gamma)$. Hence $\tilde A^{-n}A^n x = \tilde A^{-n}v + \tilde A^{-n}\gamma \in V+s'(\Gamma)$ as $\tilde A$, the restriction of $A$ to $ER(A)$,
is an isomorphism of $ER(A)$  preserving $V$ and inducing an isomorphism $\bar A$ on $\Gamma$.
So we may write $x = v'+\gamma'$ with $v'\in V$ and $\gamma'\in s'(\Gamma)$. Then $v' = \pi_V(x)$ and so we get $A^n\pi_V(x) = A^n(x) - A^n\gamma'\in\Sigma$.    
\end{proof}
 
By a result of \cite{KS} (see also \cite{BG}), the collection of eigenvalues 
of the linear transformation $A|_V$ equals the collection of all algebraic conjugates of the eigenvalues of the linear inflation $\Lambda$. Moreover, the multiplicity of any $\lambda$ as an eigenvalue of $\Lambda$ is no larger than the multiplicity of $\lambda$ as an eigenvalue of $A$.  A non-unit determinant of $\bar{A}$ implies the existence of a non-unit (and nonzero) eigenvalue of $A$ that has multiplicity greater than its multiplicity as an eigenvalue of $\Lambda$ (see example \ref{ex-3l}).

\begin{cor}\label{for split cohomology} Suppose that every eigenvalue
  of $A$ that is not an algebraic unit has the same multiplicity as an
  eigenvalue of $\Lambda$ as it does as an eigenvalue of $A$. Then the
  sequence (\ref{exact sequence 3}) splits, as in Proposition \ref{split
    cohomology}. 
\end{cor} 

\subsection{Examples of non-irreducible substitutions}
We present one-dimensional substitution tilings for which the
sequences (\ref{exact sequence 1}) and/or (\ref{exact sequence 3}) do
not split. 

Note that 
by Lemma~\ref{lem-2} the group of infinitesimal elements
$\Inf(\Omega)$ necessarily has trivial
intersection with the image of $\theta$ and so the sum
$\Inf(\Omega)+\im\theta$ is direct. 
\subsubsection{Non-splitting example}\label{ex-3l}
We consider the following substitution
$$a\mapsto abb, \quad b\mapsto aaa,$$
whose letters have equal length when realized as tiles of a one-dimensional tiling. We fix this length to be $1$ and so this fixes the action of $\R$ on the continuous hull $\Omega$.
The expansion matrix is $\Lambda=(3)$. By Host's characterization of the eigenvalues, $\beta\in E$ if and only if $\beta 3^n\: mod \:1\stackrel{n\to\infty}\longrightarrow 0$ which is clearly only possible if $\beta 3^n\in\Z$ for some $n$ and hence $E=\Z[1/3]$. It is not difficult to write down an eigenfunction $f_1$ for $\beta = 1$: take any $T_0\in\Omega$ and define $f_1(T_0-t) = \exp(2\pi\imath t)$. Given that all tiles have length $1$ the value of $f_1(T_0-t)$ depends only on the relative position of the tile on $0$ ($f$ is strongly pattern equivariant) and hence $f$ extends by continuity to $\Omega$.  
But the substitution is non-periodic and primitive and hence recognizable. This means that one can
recognize the three-letter words in any $T\in\Omega$ which arrise from a substitution of a letter
and hence also $f_{3^{-n}}(T_0-t) = \exp(2\pi\imath 3^{-n}t)$ is strongly pattern equivariant and so extends. In particular, and as it should be, $[f_{3^{-n}}] = [f_{1}\circ \Phi^n]$. 

Note that the substitution has a common prefix. 
As explained above, the first cohomology group of $\Omega$ is
therefore given by the direct limit defined by the transpose of
the incidence matrix of the substitution, which is 
\[A=\begin{pmatrix}
1 & 2 \\
3 & 0 \\
\end{pmatrix}.\]
Since this matrix is invertible over the rationals we obtain 
\begin{eqnarray*}
H^1(\Omega) \cong \bigcup_{n\in\N} A^{-n}\Z^2 &=& 
\bigcup_{n\in\N} \frac15\begin{pmatrix}
1 & -2 \\
1 & 3 \\
\end{pmatrix}
\begin{pmatrix}
3^{-n} & 0 \\
0 & (-2)^{-n} \\
\end{pmatrix}
\begin{pmatrix}3 & 2 \\
-1 & 1 \\
\end{pmatrix}
\Z^2 \\
& = &\Z[1/3]
\begin{pmatrix}
1 \\
1 \\
\end{pmatrix} +
\Z[1/2]\begin{pmatrix}
-2 \\
3 
\end{pmatrix} + \bigcup_{k=1}^4 
\begin{pmatrix}k  \\ 0 \\
\end{pmatrix}
.\end{eqnarray*}
Taking into account the isomorphism between
$[\Omega,S^1]$ and $H^1(\Omega)$ it is easily seen that $[f_{1}]$
corresponds to 
the element $\begin{pmatrix} 1 \\ 1 \\ \end{pmatrix}$ and consequently
$[f_{3^{-n}}]$  
is represented as $A^{-n}\begin{pmatrix} 1 \\ 1 \\ \end{pmatrix}=\frac1{3^n}\begin{pmatrix} 1 \\ 1 \\ \end{pmatrix}$. Thus $$\theta(E)=\Z[1/3] \begin{pmatrix} 1 \\1 \\ \end{pmatrix}.$$
We now consider the map $\tau$ which, according to the
general theory \cite{AndressRobinson} is given by the pairing
$\tau(x) = \langle\nu,x\rangle$ with the left-Perron Frobenius
eigenvector $\nu$ of $A$ normalized to $\nu_1+\nu_2 = 1$. This is $\nu
= \frac15(3,2)$.  So
\[ \tau ( A^{-n} v ) = \frac15 3^{-n} (3v_1+2v_2).\]
It follows that $$\freq(\Omega)=\im\tau = \frac15\Z[1/3]$$ and 
$$\Inf(\Omega) = \ker\tau =  \Z[1/2] \begin{pmatrix}
 -2 \\
 3 \\
 \end{pmatrix}.$$
In particular, $\im\theta$ and $\Inf(\Omega)$ are subgroups of
$H^1(\Omega)$ with trivial intersection 
but they generate only a sub-group of index $5$ in $H^1(\Omega)$.
Likewise, 
$\tau\circ\theta(E)=\Z[1/3]$ is a subgroup of index $5$ in $\im\tau$.  
\begin{prop} Neither of the exact sequences 
(\ref{exact sequence 1}) or (\ref{exact sequence 3}) splits.
\end{prop} 
\begin{proof} Suppose the sequence (\ref{exact sequence 3})
splits. There is then a $\Z$-module homomorphism $s:\bigcup_{n\in\N}
A^{-n}\Z^2\to E$ with $s\circ \theta=id$. This $s$ extends to a
$\Q$-linear map $s:\Q^2\cong\bigcup_{n\in\N}
A^{-n}\Z^2\otimes_\Z\Q\to\Q\cong\Z[1/3] \otimes_\Z\Q$. There is then a
$w\in\Q^2$ so that $s(v)=\langle w,v\rangle$ for all $v\in\Q^2$. 

Since each element of $\Z[1/2] \begin{pmatrix}
 -2 \\
 3 \\
 \end{pmatrix}$ is infinitely divisible by $2$ in $\bigcup_{n\in\N} A^{-n}\Z^2$ also $s(v)=\langle w,v\rangle$ is infinitely divisible by $2$ in $\Z[1/3]$ for each $v\in \Z[1/2] \begin{pmatrix}
 -2 \\
 3 \\
 \end{pmatrix}$. This means that $w$ is orthogonal to $\Z[1/2] \begin{pmatrix}
 -2 \\
 3 \\
 \end{pmatrix}$, say $w=t \begin{pmatrix}
 3 \\
 2 \\
 \end{pmatrix}$ with $t\in\Q$. Then $s(\begin{pmatrix}
 1 \\
 0 \\
 \end{pmatrix}) = \langle t \begin{pmatrix}
 3 \\
 2 \\
 \end{pmatrix},\begin{pmatrix}
 1 \\
 0 \\
 \end{pmatrix}\rangle=3t$ must lie in $\Z[1/3]$; that is, $t\in\Z[1/3]$. Let $t=t_0/3^{n_0}$ with $t_0,n_0\in\Z$. Now the restriction of $s$ to $\theta(E)=\Z[1/3] 
 \begin{pmatrix}1 \\1  \end{pmatrix}$ is surjective (since $s\circ\theta=id$) so there is an $x\in\Z[1/3] \begin{pmatrix}1 \\1  \end{pmatrix}$, say $x=x_0/3^{m_0} \begin{pmatrix}1 \\1  \end{pmatrix}$, with $x_0,m_0\in\Z$, so that $s(x \begin{pmatrix}
 1 \\
 1 \\
 \end{pmatrix})=\langle t \begin{pmatrix}
 3 \\
 2 \\
 \end{pmatrix},x\begin{pmatrix}
 1 \\
 1 \\
 \end{pmatrix}\rangle=(t_0/3^{n_0})(x_0/3^{m_0})5=1$. But this implies that $5$ divides $3$.
 
 The argument for the sequence (\ref{exact sequence 1}) is completely similar, one only has to interchange the role of the eigenvectors of $A$.
\end{proof}

\subsubsection{Period doubling}\label{period doubling}
We consider the following substitution
$$a\mapsto ab, \quad b\mapsto aa,$$
whose letters have equal length when realized as tiles of a one-dimensional tiling. We fix this length to be $1$ and so this fixes the action of $\R$ on the continuous hull $\Omega$.
The expansion matrix is $\Lambda=(2)$. As above one sees that  $E=\Z[1/2]$ and that $f_{2^{-n}}(T_0-t) = \exp(2\pi\imath 2^{-n}t)$ is strongly pattern equivariant and so extends to an eigenfunction to eigenvalue $2^{-n}$. 

The transpose of  the incidence matrix of the substitution is 
\[A=\begin{pmatrix}
1 & 1 \\
2 & 0 \\
\end{pmatrix}.\] 
Since the substitution has a common prefix  we get
\begin{eqnarray*} 
H^1(\Omega) \cong \bigcup_{n\in\N} A^{-n}\Z^2 &=& \bigcup_{n\in\N} \frac13\begin{pmatrix}
1 & -1 \\
1 & 2 \\
\end{pmatrix}
\begin{pmatrix}
2^{-n} & 0 \\
0 & (-1)^{-n} \\
\end{pmatrix}
\begin{pmatrix} 2 & 1 \\
-1 & 1 \\
\end{pmatrix}
\Z^2\\ 
&=& \Z[1/2]
\begin{pmatrix}
1 \\
1 \\
\end{pmatrix} +
\Z\begin{pmatrix}
-1 \\
2 
\end{pmatrix} + \bigcup_{k=0}^2
\begin{pmatrix}k  \\ 0 \\
\end{pmatrix}.
\end{eqnarray*}

As above one sees that  $[f_{2^{-n}}]$ 
is represented as $A^{-n}\begin{pmatrix} 1 \\ 1 \\ \end{pmatrix}=\frac1{2^n}\begin{pmatrix} 1 \\ 1 \\ \end{pmatrix}$. Thus $$\theta(E)=\Z[1/2] \begin{pmatrix} 1 \\1 \\ \end{pmatrix}.$$
We now consider the map $\tau$.  The left-Perron Frobenius eigenvector of $A$ normalized to $\nu_1+\nu_2 = 1$ is $\nu = \frac13(2,1)$. So
\[ \tau ( A^{-n} v ) = \frac13 2^{-n} (2v_1+1v_2).\]
It follows that 
$$\freq(\Omega) = \frac13\Z[1/2]$$ and $$\Inf(\Omega) =  \Z \begin{pmatrix}
 -1 \\
 2 \\
 \end{pmatrix}.$$ Hence $\im\theta+\Inf(\Omega)$ is a subgroup of
 index $3$ in $H^1(\Omega)$.
\begin{prop} The exact sequence (\ref{exact sequence 3}) splits and hence
$H^1(\Omega) \cong \Z[1/2] \oplus \Z$. This splitting does
not respect the order as the infinitesimal elements form an index $3$
subgroup of the second summand.
\end{prop}
\begin{proof}
$\coker\theta$ is given by the quotient $ \bigcup_{n\in\N} A^{-n}\Z^2/\sim$ where $x\sim y$ if there exist $k,n$ such that $x-y = \frac{k}{2^n} \begin{pmatrix} 1 \\1 \\ \end{pmatrix}$. It follows that
$ \begin{pmatrix} 3 \\0 \\ \end{pmatrix} \sim   \begin{pmatrix} -1 \\2 \\ \end{pmatrix}$,
$ A^{-1} \begin{pmatrix} 1 \\0 \\ \end{pmatrix} =  \begin{pmatrix} 0
  \\1 \\ \end{pmatrix} \sim - \begin{pmatrix} 1 \\0 \\ \end{pmatrix}$,
and $ A^{-2} \begin{pmatrix} 1 \\0 \\ \end{pmatrix} =
\frac12\begin{pmatrix} 1 \\ -1 \\ \end{pmatrix} \sim \begin{pmatrix} 1
  \\0 \\ \end{pmatrix}$. Thus
$\coker\theta$ is thus generated by the equivalence class of the
element $\begin{pmatrix} 1 \\0 \\ \end{pmatrix}$. In particular it is
finitely generated and thus the sequence splits. 

Now it is clear that the infinitesimal elements form an index $3$
subgroup of the second summand. The only possible orderings on $\Z$
are the trivial order, in which case all elements are infinitesimal,
or the standard order, in which only $0$ is
infinitesimal. The above splitting of $H^1(\Omega)$ is thus not an
order preserving 
splitting into a direct sum of ordered groups.
\end{proof}
\begin{prop}
The exact sequence (\ref{exact sequence 1}) does not split.
\end{prop}
\begin{proof} The proof is like the above:
Suppose the sequence $0\to \Z \stackrel{\psi}{\to} \bigcup_{n\in\N} A^{-n}\Z^2 \to \frac13 \Z[1/2]\to 0$ splits, where $\psi(1) = \begin{pmatrix}1 \\-2 \end{pmatrix}$.
There is then a $\Z$-module homomorphism $s:\bigcup_{n\in\N} A^{-n}\Z^2\to \Z$ with $s\circ \psi=id$. This $s$ extends to a $\Q$-linear map $s:\Q^2\cong\bigcup_{n\in\N} A^{-n}\Z^2\otimes_\Z\Q\to\Q\cong\Z \otimes_\Z\Q$. 
There is then a $w\in\Q^2$ so that $s(v)=\langle w,v\rangle$ for all $v\in\Q^2$.

Since each element of $\Z[1/2] \begin{pmatrix} 1 \\ 1 \end{pmatrix}$ is infinitely divisible by $2$ in $\bigcup_{n\in\N} A^{-n}\Z^2$ also $s(v)=\langle w,v\rangle$ is infinitely divisible by $2$ in $\Z$ for each $v\in \Z[1/2] \begin{pmatrix}1 \\1  \end{pmatrix}$. This means that $w$ is orthogonal to $\begin{pmatrix}1 \\1 \end{pmatrix}$, say $w=t \begin{pmatrix}
 1 \\-1  \end{pmatrix}$ with $t\in\Q$. Then $s(\begin{pmatrix}
 1 \\
 0 \\
 \end{pmatrix}) = \langle t \begin{pmatrix}
 1 \\
-1 \\
 \end{pmatrix},\begin{pmatrix}
 1 \\
 0 \\
 \end{pmatrix}\rangle=t$ must lie in $\Z$; that is, $t\in\Z$. 
  Now the restriction of $s$ to $\psi(\Z)=\Z 
 \begin{pmatrix}1 \\-2  \end{pmatrix}$ is surjective so there is an $x\in\Z$ such that
 $s(x \begin{pmatrix}
 1 \\
 -2 \\
 \end{pmatrix})=\langle t \begin{pmatrix}
 1 \\
 -1 \\
 \end{pmatrix},x\begin{pmatrix}
 1 \\
 -2 \\
 \end{pmatrix}\rangle=tx3=1$. But this implies that $1/3$ is an integer.
\end{proof}
\subsection{Thue-Morse}
\renewcommand{\o}{\bar{1}}
We finally consider the Thue-Morse substitution
$$1\mapsto 1\o, \quad \o\mapsto \o1,$$
whose letters have equal length when realized as tiles of a one-dimensional tiling. We fix this length to be $1$ and so this fixes the action of $\R$ on the continuous hull $\Omega$.
The expansion matrix is $\Lambda=(2)$. 

The substitution does not force its border and so we use the technique of collared tiles (the bracketed tile is the actual tile, the other two are the collar):
$$a := 1(\o)1,\quad b:= \o(\o)1,\quad c:=1(\o)\o,\quad \bar{a} := \o(1)\o,\quad \bar{b}:= 1(1)\o,
\quad \bar{c}:=\o(1)1.$$
The Anderson-Putnam complex $\Gamma$ has $6$ edges, namely the collared tiles which we orient to the right in the tiling, and $4$ vertices $v,\bar{v},w,\bar{w}$. Indeed $w$ is a vertex at the end of $c$ and the beginning of $\bar{b}$ (and  $\bar w$ is a vertex at the end of $\bar c$ and the beginning of ${b}$), and $v$ is a vertex at the end of $b$, $\bar a$ and the beginning of $a$, $c$ (and  $\bar v$ is a vertex at the end of $\bar b$, $a$ and the beginning of $\bar a$, $\bar c$).
Its cohomology is thus that of the complex
$$ 0 \to \Z^4 \stackrel{\delta^T}\to \Z^6 \to 0$$
where\footnote{we use the bases: $a$, $b$, $c$, $\bar a$, $\bar b$, $\bar c$, and
$v$, $\bar v$, $w$, $\bar w$.}
$$ \delta =  \begin{pmatrix}
 1 & 1 & 0 & -1 & 0 & -1 \\
 -1 & 0 & -1 & 1 & 1 & 0 \\
 0 & -1 & 1 & 0 & 0 & 0 \\
 0 & 0 & 0 & 0 & -1 & 1
 \end{pmatrix}.$$
In particular, $H^1(\Gamma) \cong \Z^3$. We now have to determine the matrix
$A$ corresponding to the endomorphism induced by the substitution on $H^1(\Gamma)$. The latter reads on collared tiles as follows: 
$$a \mapsto b\bar{c},\quad b\mapsto a\bar{c},\quad c\mapsto b\bar{a},
\quad \bar{a} \mapsto \bar{b}c,\quad 
\bar{b}\mapsto \bar{a}c,\quad \bar{c}\mapsto\bar{b}a$$ 
which has incidence matrix $\sigma_{ij}$ (equal to the number of tiles of type $i$ in the supertile of type $j$)
$$\sigma =  \begin{pmatrix}
 0 & 1 & 0 & 0 & 0 & 1 \\
 1 & 0 & 1 & 0 & 0 & 0 \\
 0 & 0 & 0 & 1 & 1 & 0 \\
 0 & 0 & 1 & 0 & 1 & 0 \\
 0 & 0 & 0 & 1 & 0 & 1 \\
 1 & 1 & 0 & 0 & 0 & 0
 \end{pmatrix}.$$
We find that the left-eigenvectors  of $\sigma$ which are not in $\im\delta^T$ are,
$  \begin{pmatrix} 1 & 1 & 1 & 1 & 1 & 1 \end{pmatrix}$ to eigenvalue $2$,
$  \begin{pmatrix} 1 & 1 & 1 & -1 & -1 & -1 \end{pmatrix}$ to eigenvalue $0$, and
$  \begin{pmatrix} 1 & -1 & 0 & 1 & -1 & 0 \end{pmatrix}$ to eigenvalue $-1$.
It follows that the eventual range of the endomorphsim $A$ induced by $\sigma^T$ on $\Z^6/\im\delta^T$ has dimension $2$ and that the restriction $\tilde A$ of $A$ to its essential range is obtained by row-reducing the above left-eigenvectors of $\sigma$ to eigenvalues $2$ and $-1$ w.r.t.\ the rows spanning $\im\delta^T+\ker\sigma^T$.
 The result is
 $  \begin{pmatrix} 2 & 4 & 0 & 0 & 0 & 0 \end{pmatrix}$, the left-eigenvector of $\sigma$ modulo $\langle\im\delta^T,\ker\sigma^T\rangle$ to eigenvalue $2$, and
$  \begin{pmatrix} 2 & -2 & 0 & 0 & 0 & 0 \end{pmatrix}$ the one to eigenvalue $-1$. It follows that
\[\tilde A=\frac13 \begin{pmatrix}
1 & 1 \\
2 & -1 \\
\end{pmatrix}\begin{pmatrix}
2 & 0 \\
0 & -1 \\
\end{pmatrix}\begin{pmatrix}
1 & 1 \\
2 & -1 \\
\end{pmatrix}=\begin{pmatrix}
0 & 1 \\
2 & 1 \\
\end{pmatrix}.\]
Hence
\[H^1(\Omega) \cong \bigcup_n \tilde A^{-n}\Z^2 = \Z[1/2]
\begin{pmatrix}
1 \\
2 \\
\end{pmatrix} +
\Z\begin{pmatrix}
1 \\
-1 
\end{pmatrix} + \bigcup_{k=1}^2
\begin{pmatrix}k  \\ 0 \\
\end{pmatrix} .\]
As above one sees that  $E=\Z[1/2]$ and that $f_{2^{-n}}(T_0-t) = \exp(2\pi\imath 2^{-n}t)$ is strongly pattern equivariant and so extends to an eigenfunction to eigenvalue $2^{-n}$. Furthermore the eigenfunction $f_1$ represents the class in $H^1(\Omega)$ given by the class of $(1,1,1,1,1,1)^T$ modulo $\langle\im\delta^T,\ker\sigma^T\rangle$
which corresponds to $(2,4)^T$ in 
$\bigcup_n \tilde A^{-n}\Z^2$. Hence
$$ \theta(E) = \Z[1/2]
\begin{pmatrix}
1 \\
2 \\
\end{pmatrix} .
$$
More generally, since $\tau(v)=0$ for any left-eigenvector of $\sigma$ to an eigenvalue different to $2$ the map $\tau:\bigcup_n \tilde A^{-n}\Z^2 \to \R$ is given by $\tau((2,4)^T) = 1$ and $\tau((1,-1)^T)=0$.
It follows that \[ \tau ( A^{-n} v ) = \frac16 2^{-n} (v_1+v_2),\]
 $\freq(\Omega) = \frac13\Z[1/2]$, and $\Inf(\Omega) =  \Z \begin{pmatrix}
 1 \\
 -1 \\
 \end{pmatrix}$.
 Now the same calculation as for the period doubling sequence yields:
\begin{prop} The exact sequence (\ref{exact sequence 3}) splits and hence
$H^1(\Omega) \cong \Z[1/2] \oplus \Z$. This splitting does not respect
the order as infinitesimal elements form an index $3$ subgroup of the
second summand.
The exact sequence (\ref{exact sequence 1}) does not split.
\end{prop}
\subsection{The action of the substitution on first cohomology}
As a final observation, let us consider complex valued cohomology looking at possibly complex eigenvalues of the action of $\Phi^*$ on the first cohomology. 
It is shown in \cite{BG} that all eigenvalues of $\Lambda$ are also eigenvalues of $\Phi^*$.
\begin{theorem}
The image of $\tau$ is invariant under the action of $\Lambda^T$. Moreover,
suppose that $\Phi^* x = \lambda x$ for some cohomology element $x\in H^1(\Omega,\mathbb C)$. If $\tau(x)\neq 0$ then $\lambda$ must be an eigenvalue of $\Lambda$. 
\end{theorem}  
\begin{proof}
Note that, by unique ergodicity,
\[ C_\mu (\alpha) = \lim_{k\to \infty} \frac1{\mu(I_k)} \int_{I_k} \alpha(T_0-x) d\nu(x)\]
where $I_k = [-k,k]^n$ is the cube of side length $2k$ centered at $0$ and $\nu$ is the 
Lebesgue measure. This holds for any $T_0$.  
We have $\Phi^*(f^{-1}df)(T_0-x) = \Lambda^T f^{-1}(\Phi(T_0)-\Lambda x) df(\Phi(T_0)-\Lambda x)$. Hence $C_\mu (\Phi^*(f^{-1}df))= \Lambda^TC_\mu (f^{-1}df)$ implying 
 that $\tau([f\circ \Phi]) = \Lambda^T\tau([f])$. So if $x$ is an eigenvector of $\Phi^*$ to eigenvalue $\lambda$ and $\tau(x)\neq 0$ then $\tau(x)$ is an eigenvector of 
$\Lambda^T$ to $\lambda$.
\end{proof}

 Thus, if $x$ is in a generalized eigenspace of $\Phi^*$ corresponding to an eigenvalue that is not also an eigenvalue of $\Lambda$, then $x\in \ker(\tau)$. In the case $n=1$, such $x\in [\Omega,S^1]$
must be infinitesimal (see Subsection \ref{order structure} and example \ref{period doubling}).

\end{document}